\documentclass[11pt,authoryear]{article}
\usepackage[round]{natbib}
\usepackage{amsmath, amsthm, amssymb, amsfonts}
\usepackage{graphicx}
\usepackage{epstopdf}
\usepackage{epsfig,color}
\usepackage[utf8]{inputenc}
\usepackage[margin=3cm]{geometry}

\newcommand{\tr}{^{\prime}}
\def\bd#1{\mbox{\boldmath $#1$}}
\newcommand{\diag}{{\rm diag}}
\def\bl#1{\mbox{\footnotesize \boldmath {$#1$}}} 
\def\m#1{\mbox{#1}}                
\def\cg#1{\mbox{${\cal #1}$}}      

\newcommand{\ci}{\mbox{\protect $\: \perp \hspace{-2.3ex}\perp$ }}

\newtheorem{lemma}{Lemma}
\newtheorem{definition}{Definition}
\newtheorem{proposition}{Proposition}
\newtheorem{example}{Example}
\newtheorem{remark}{Remark}
\author{Antonio Forcina\\ Dipartimento di Economia,
University of Perugia,Italty }
\title{Multiplicative models for frequency data, estimation and testing}
{\markboth{A. Forcina}{Multiplicative models}

\begin{document}
\maketitle
\begin{abstract}
This paper is about models for a vector of probabilities whose elements must have a multiplicative structure and sum to 1 at the same time; in certain applications, as basket analysis, these models may be seen as a constrained version of quasi-independence. After reviewing the basic properties of these models, their geometric features as a curved exponential family are investigated. A new algorithm for computing maximum likelihood estimates is presented and new insights are provided on the underlying geometry. The asymptotic distribution of three statistics for hypothesis testing are derived and a small simulation study is presented to investigate the accuracy of asymptotic approximations.
\end{abstract}
\paragraph{Keywords.}
Relational models,  Curved exponential families, Mixed parametrization, Iterative proportional fitting, Log-linear models, Multinomial distribution
\section{Introduction}
The class of Relational models, introduced by \cite{KRD12} and developed further by \cite{KRipf1} and \cite{KRJMA16}, generalize log-linear models for contingency tables under Poisson or multinomial sampling in several directions. Essentially, the cells of the table to which the models can be applied, do not need to be produced by the cross classification of a set of categorical random variables and may be determined by an arbitrary underlying structure. An obvious instance is when attention is restricted to a specific subset of a cross classification by conditioning or because the table has structural zeros.

The present paper concentrates on what is the most intriguing subclass of relational models: multiplicative models for the vector of probabilities defining a multinomial distribution. Perhaps the main application is the independence model of \cite{AitchSilvey60} applied to basked analysis, where the probabilities of certain cells must be equal to the product of the probabilities of other cells, a relation which conflicts with the requirement that all probabilities have to sum to 1.
Imposing both constraints, that is multiplicative structure and summing to 1, leads to a curved exponential family within the multinomial distribution. The resulting models have some peculiar features, the main one being that, though the vector of mean parameters (in the sense of \cite{Barndorff1978}) of the fitted model is still proportional to the vector of sufficient statistic, as in an ordinary log-linear model, the constant of proportionality is not equal to the sample size. This result, which appeared in  \cite{KRD12}, provides a key to computation and interpretation of maximum likelihood estimated (MLEs).

This paper investigates multiplicative model as a curved exponential family and describes a new procedure for computing the MLEs which exploits advanced results on the multinomial as an exponential family; this formulation provides several new insights both on the geometry of the curved exponential family and on that of the MLE.
In Section 2, after introducing notations and revising the basic properties of the models, their interpretation and geometry is examined in some depth. A new algorithm for computing MLEs is described in Section 3 where certain geometric properties of the MLE are also highlighted. The asymptotic distribution of three alternative statistics which may be used for testing are discussed in section 4 and a small simulation study is presented to investigate the accuracy of the asymptotic approximations for finite sample sizes.

%
\section{Notations and preliminary results}
Let $V$ denote a discrete random variable with outcomes in $\cg V$ = $\{v_1,\dots ,v_r\}$, let $\pi_j=P(V=v_j)$ and $\bd\pi$ denote the vector with elements $\pi_j$, $j=1,\dots ,r$; assume that $\bd\pi$ belongs to $\cg P$ = $\{\bd \pi:\:\pi_j>0,\: \bd 1\tr \bd\pi = 1\}$. Let $Mn(n,\bd\pi)$ denote a multinomial distribution with sample size $n$ and assume that the objective is to make inference about $\bd\pi$ based on the vector of observed frequencies $\bd y$ $\sim$ $Mn(n,\bd\pi)$.

Let $\bd X$ be an $r \times k$ design matrix whose elements are 0 or 1. Assume that $\bd X$ is of full column rank and that the unitary vector $\bd 1$ = $(1, \dots,1)\tr$ does not belong to the space spanned by the columns of $\bd X$; in the terminology of \cite{KRipf1} these are models for probabilities without the overall effect. Assume also that $\bd X\bd 1 \geq \bd 1$, otherwise the model would be incompatible with the definition below, as explained in Remark 1.
\begin{definition}
The vector $\bd\pi$ satisfies a multiplicative model, denoted $\bd\pi\in \cg M(\bd X)$, if
\begin{equation}\label{genll}
\log \bd{\pi} = \bd X \bd \theta,
\end{equation}
where $\bd \theta$ is a vector of $k$ log-linear parameters.
\end{definition}
Multiplicative models under multinomial sampling belong to the class of Relational Models introduced by \cite{KRD12}.
\begin{remark}
If $\bd X$ contained a row of 0's, the corresponding element of $\bd\pi$ would be equal to 1, which conflicts with the requirement that $\bd\pi\in \cg P$. Because the first row in a design matrix $\bd X$ associated with a cross classification under the corner point coding is a row of 0s, multiplicative models can be only applied to incomplete tables where the initial cell (and possibly others) has been removed.
\end{remark}

From linear algebra (\ref{genll}) holds if and only if, any $(r-k)\times r$ matrix $\bd C$ which is of full row rank and such that $\bd C\bd X=\bd 0$, satisfies: $ \bd C \log \bd\pi$ = $\bd 0$. The Proposition below, which could also be derived from \cite{KRD12}, shows that, for any $\cg M(\bd X)$, the constraint matrix $\bd C$ can always be written in a kind of canonical form:
\begin{proposition}\label{oneORtheorem}
Possibly after a non singular linear transformation, the constraint matrix $\bd C$ can be written in the form
$$
\bd C\tr =
\begin{pmatrix} \bd c & \bd H\tr \end{pmatrix}
$$
where the vector $\bd c$ is such that $\bd c\tr\bd 1$ = $-1$ and $\bd H$ is a $(r-k-1)\times r$ matrix of row contrasts, that is $\bd H\bd 1$ = $\bd 0$.
\end{proposition}
\begin{proof}
Because $\bd 1$ does not belong to the space generated by the columns of $\bd X$, $\bd C$ must contain at least a row which does not sun to 0; because rows can be permuted, there is no loss of generality in assuming that the first row of $\bd C$, say $\bd c_1\tr$, is such that $\bd c_1\tr\bd 1$ = $t_1\neq 0$. To replace any other row, say $\bd c_a$, with $\bd h_a$ = $\bd c_a - \bd c_1(t_a/t_1)$ is equivalent to left multiply $\bd C$ by a lower triangular matrix which has $1$s on the main diagonal and is thus non-singular. Finally let $\bd c$ = $-\bd c_1/t_1$.
\end{proof}
\subsection{Examples of multiplicative models}
Example 4.1 in \cite{KRD12} describes a multinomial distribution with three cells which, under the given context, has a multiplicative structure; their Example 1.2 instead deals with independence in an incomplete $2\times 2$ table.  An example of a multiplicative model in an incomplete $2^3$ table is given in Example 3 in \cite{AitchSilvey60}; the same example is revisited by \cite{KRJMA16} in their Example 2.1.

It may be worth examining the latter model in more detail in the context of basket analysis; suppose we restrict attention to baskets containing only three specific items, with the empty basket corresponding to a structural zero. Formally, we have the cross classification of three binary variables without the configuration "0,0,0". The model of complete independence in this table is an instance of a quasi-independence (QI) model.  According to \cite{KRJMA16}, in certain contexts, the QI model may not represent adequately the assumption of real interest and suggest considering the constraints proposed by \cite{AitchSilvey60} in their Example 3 which require a multiplicative structure even within incomplete  2-way tables; they call this model AS independence. To clarify the issue, consider Table \ref{T0} where the structure of probabilities satisfying QI and the corresponding AS independence are displayed.
\begin{table}[h]
\caption{ \label{T0} \it Probabilities in an incomplete 3-way table under QI and AS; $a,\:b,\:c$ are parameters and $p,\:q\:,r$ are probabilities}
\begin{center}
\begin{tabular}{lccccccc} \hline
Model & $p_{001}$ & $p_{010}$ & $p_{011}$ & $p_{100}$ & $p_{101}$ & $p_{110}$ & $p_{111}$ \\ \hline
QI & $cs$    & $bs$        & $bcs$      & $as$      & $acs$    &  $abs$      & $abcs$ \\
AS & $r$ & $q$ & $qr$ & $p$ & $pr$ & $pq$ & $pqr$
\\  \hline
\end{tabular}
\end{center}
\end{table}
Essentially, QI forces independence within the 3 complete two-way tables obtained by fixing, in turn, one of the three variable to 1. The $s$ parameter appearing under the QI model in Table \ref{T0} is simply a scaling constant which makes probabilities sum to 1. However, just because of this scaling factor, the multiplicative structure expected under AS independence is violated within each of the three incomplete 2-way tables obtained by conditioning, in turn, each variable to 0.
%
\subsection{Multiplicative models as a curved exponential family}
To see that the multinomial distribution $Mn(n,\bd\pi)$ with $\bd \pi\in \cg M(\bd X)$ is a curved exponential family \citep[Theorem 3.2,][]{KRD12}, note that the model could be imposed in two steps: first assume that
\begin{equation}
\log\bd \pi = \bd X\bd\theta-\bd 1\log(\bd 1\tr\exp(\bd X\bd\theta));
\label{eqEFreg}
\end{equation}
this is a log-linear model defined by the linear restriction $\bd \eta$ = $\bd H\log\bd \pi$ = $\bd 0$ on the canonical parameter. Let $\tilde{\bd\pi}$ denote the MLE under this preliminary model. Next impose the additional constraint $\bd c\tr\log\tilde{\bd\pi}$ = 0. Because $\bd c\tr\bd X$ = $\bd 0\tr$ and $\bd c\tr\bd 1=-1$, this is equivalent to the non linear constraint
\begin{equation}
F(\bd\theta) = \log(\bd 1\tr\exp(\bd X\bd\theta)) = 0;
\label{Implicit}
\end{equation}
let $\cg F(\bd X)$ be the set of $\bd\theta$ which satisfy (\ref{Implicit}). Because $F(\bd\theta)$ is continuous and differentiable, the theory of implicit functions, see for instance \cite{Courant},  insures that, in a neighbourhood of any given  $\bd\theta_0\in\cg F(\bd X)$, the constraint defines a smooth surface which can be parameterised by a $(k-1)$ dimensional vector $\bd\phi$, whose points satisfy (\ref{Implicit}).  To understand the geometry of $\cg F(\bd X)$ better, note that, while any $\bd\theta\in \Re^k$ satisfies (\ref{eqEFreg}),  because $\log \bd\pi$ = $\bd X\bd\theta$, it follows that $\cg F(\bd X) \subset \cg C$, the convex cone defined by $-\bd X\bd\theta  > \bd 0$. Let $\bd U$ be the matrix whose $g$ columns are the edges (generators) of $\cg C$, then an element of $\cg F(\bd X)$ may be generated at random as follows:
\begin{enumerate}
\item
sample the $g$ elements of $\bd q$ from a uniform in (0,1) and then scale them to sum to 1,
then $\bd u$ = $\bd U\bd q$ $\in\cg C$;
\item
find $c$ such that $\log[\bd 1\tr\exp(\bd X\bd uc)]$ = 0, then $\bd\theta$ = $c\bd u\in \cg F(\bd X)$.
\end{enumerate}
In words, for any $\bd u\in \cg C$, there exists $c>0$ such that $\bd\theta$ = $c\bd u\in \cg F(\bd X)$.
\begin{example}
For $r=4$, consider the design matrix given below together with the matrix $\bd U$ whose columns are the edges of the convex cone $\cg C$
$$
\bd X\tr = \begin{pmatrix}
1 & 1 & 1 & 0 \\
0 & 1 & 1 & 1 \\
0 & 0 & 1 & 1  \end{pmatrix}, \quad
\bd U = \begin{pmatrix}
     0  &   0 &   -1\\
     0  &  -1 &    1\\
    -1  &   1 &   -1 \end{pmatrix}.
$$
For a random sample of 6,000 vectors in $\cg C$, the corresponding elements of $\cg F(\bd X)$ are plotted in Figure \ref{Fig1}; points with coordinates greater than 10 in absolute value are omitted to improve visibility. For any $\bd\theta\in \cg F(\bd X)$, the corresponding point with coordinates $\bd\tau$ = $\bd X\tr\bd\pi$ = $\bd X\tr \exp(\bd X\bd\theta)$ are also plotted. The set $\cg F(\bd X)$ appears to be highly non linear while the curvature in the space of $\bd\tau$ is moderate.
\begin{figure}
\centering
\makebox{\includegraphics[width=14cm,height=6cm]{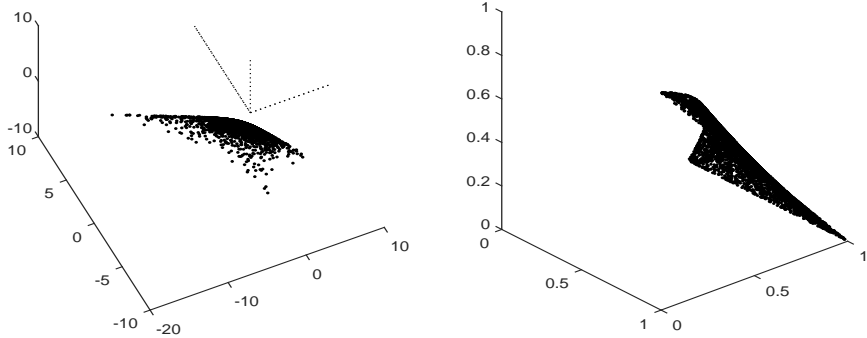}}
\caption{\label{Fig1}A random sample of $\bd\theta\in \cg F(\bd X)$ spanning most of $\cg F(\bd X)$ is plotted on the left panel together with the axis of the positive orthant. The corresponding values of $\bd\tau$ = $\bd X\tr\exp(\bd X\bd\theta)$ are plotted on the left panel.}
\end{figure}
\end{example}

For a given $\bd\theta\in \cg F(\bd X)$, differentiate (\ref{Implicit}) with respect to $\bd\phi$ by the chain rule
$$
\frac{\partial F(\bd\theta)}{\partial \bd\phi} =
\frac{\partial \bd\theta\tr}{\partial\bd\phi} \frac{\partial F(\bd\theta)}{\partial\bd\theta}
=  \m{ (say) } \bd A\tr \bd X\tr \bd\pi = \bd 0,
$$
thus, for $\bd\pi\in \cg M(\bd X)$, $\bd \tau$ = $\bd X\tr\bd\pi$  must be orthogonal to all the $k-1$ columns of $\bd A$, a relation which has important implications for the geometry of MLE discussed in section 3.3.
\section{Maximum likelihood estimation}
Though MLEs for multiplicative models might be computed by general purpose algorithms like the one described by \cite{EvansForcina11},  the approach described by \cite{KRipf1} in the context of Relational models, with suitable adjustments, can be made roughly, equally efficient and provides deeper insights into the nature of these models. The objective of this section is to propose a formal framework for a dual version of the algorithm in \cite{KRipf1}; both algorithms may be seen as repeated applications of the mixed parametrization applied to the Multinomial distributions. This suggests improvements which throw additional light into the geometry of these models.

Let $\bd y$ be an $r\times 1$  vector of frequencies from $Mn(n,\bd\pi)$, let $\bd p$ = $\bd y/n$ be the vector of sample proportions and $\hat{\bd\pi}\in \cg M(\bd X)$ be the vector of MLEs.
The basic result in the Proposition below could also be derived from Theorem 3.3 in \cite{KRD12}.
\begin{proposition}
The likelihood equation can be written in the form
\begin{equation}
\hat\gamma \bd X\tr\bd p = \hat\gamma \bd X\tr\tilde{\bd\pi} = \bd X\tr\hat{\bd\pi},
\label{eqajfac}
\end{equation}
for a suitable constant $\hat\gamma>0$,
\end{proposition}
\begin{proof}
Start from the multinomial likelihood with $\log \bd\pi$ as in (\ref{eqEFreg}) and use the Lagrange multiplier $\alpha$ to account for the additional constraint $\log(\bd 1\tr e^{\bl X\bl\theta})$ = 0. To maximize the scaled lagrangian
$$
L(\bd y;\bd\pi,\alpha)/n = \bd p\tr\bd X\bd\theta - \log(\bd 1\tr\exp(\bd X\bd\theta)) - \alpha(\log(\bd 1\tr\exp(\bd X\bd\theta)) - 0),
$$
differentiate with respect to $\bd\theta$ to obtain
$$
\bd X\tr\bd p - (1+\hat\alpha) \bd X\tr\hat{\bd\pi} = \bd 0.
$$
Now divide both sides by  $\hat\gamma$ = $1/(1+\hat\alpha)$; the fact that $\hat\gamma>0$ follows from \cite{KRD12} who call it the adjustment factor. The first equality is simply to remind that $\bd X\tr\bd p$ = $ \bd X\tr\tilde{\bd\pi}$.
\end{proof}
\begin{remark}
\label{Ralpha}
In the very unlikely event that $\bd 1\tr\exp(\bd X\tilde{\bd\theta})$ = 1, $\hat\gamma=1$, meaning that the normalization factor which appears in the right hand side of  (\ref{eqEFreg}) is redundant.
\end{remark}

The main results of this section relay on the notion of the mixed parametrization for regular exponential families \cite[p. 121-123]{Barndorff1978}, \cite[5.7]{PaSa}. A brief summary of the relevant properties tailored to the given context is provided below
\subsection{The mixed parametrization}
Suppose that the $q\times 1$ vector $\bd v$ follows an exponential family (EF) distribution with canonical parameter $\bd\lambda$ and cumulant generating function $K(\bd\lambda)$. Suppose one is interested to put constraints on a suitable linear transformation of the canonical parameter; a formal treatment, partly similar to the one given below, is in \cite[Section 5]{KRD12}. Let $\bd X,\:\bd C$ be $q\times k$ and $(q-k)\times q$ respectively, of full rank and such that  $\bd C\bd X$ = $\bd 0$; let also $\bd X^-$ = $(\bd X\tr\bd X)^{-1}\bd X\tr$ and $\bd C^-$ = $\bd C\tr(\bd C\bd C\tr)^{-1}$ and define $\bd\theta$ = $\bd X^-\bd\lambda$ and $\bd\eta$ = $\bd C\bd\lambda$.  Simple calculations show that the kernel in the log of the probability distribution may be expanded as
$$
\bd v\tr\bd\lambda -K(\bd\lambda) = \bd v\tr\bd X\bd\theta + \bd v\tr \bd C^-\bd\eta - \tilde K(\bd\theta,\bd\eta).
$$
Let $\bd\mu$ = $E(\bd v)$ and $\bd \tau$ = $\bd X\tr\bd\mu$, this is the expected value of the sufficient statistic for $\bd\theta$ and is known as the {\it mean parameter} component. The pair $(\bd\tau,\:\bd\eta)$ provide a mixed parametrization and has the following properties: (i) the mapping between  $(\bd\tau,\:\bd\eta)$ and $\bd\lambda$ is one to one and differentiable; (ii) the two components $(\bd\tau,\:\bd\eta)$ are variation independent. Note, however, that $\bd\tau$ may have to satisfy specific order constraints induced by $\bd X$.

It is well known that the MLE for a sample of $n$ observation $\bd v_1,\dots , \bd v_n$ under the linear constraint $\bd\eta$ = $\bd 0$ satisfies the likelihood equation
$$
\bd t = \bd X\tr\sum \bd v_i/n = \bd X\tr\bd\mu = \bd\tau;
$$
It can be easily verified that, solving the above equation, is equivalent to apply a mixed parametrization where $\bd\tau$ = $\bd t$ and $\bd\eta=\bd 0$ which induces a mapping between $\bd \tau$ and $\bd\theta$; again this is one to one and differentiable.
\subsection{An algorithm for computing the MLE}
In the notation of Proposition 1, the algorithm described in \cite{KRipf1}, consists in two steps: (i) for fixed $\gamma$, use an iterative proportional fitting (IPF) to solve the equation $\gamma\bd X\tr\bd p$ = $\bd X\tr\bd\pi(\gamma)$ under the full set of constraints $\bd\zeta$ = $\bd C\log\bd\pi(\gamma)$ = $\bd 0$; (ii) search for the value $\hat\gamma$ such that $\bd 1_r\tr\bd\pi(\hat\gamma)$ = 1. A kind of dual procedure may consist of the following: (i) for a given $\bd\gamma$ use IPF to solve $\bd\tau$ = $\gamma\bd X\tr\bd p$ = $\bd X\tr\bd\pi(\gamma)$ under $\bd\eta$ = $\bd H \log\bd\pi (\gamma)$ = $\bd 0$, that is by ignoring the additional constraint $f(\gamma)$ = $\bd c\tr \log\bd\pi$ = 0; (ii) Search for the $\hat \gamma$ such that the function $f(\gamma)$  is equal to 0.

To see that solving (i) in \cite{KRipf1} is equivalent to apply a mixed parametrization to a Poisson distribution with $\bd\tau$ = $\gamma\bd X\tr\bd p$ and $\bd\zeta$ = $\bd 0$, note that the unconstrained canonical parameter has dimension $r$ and the coefficients of any log-linear constraints do not need to sum to 0. This explains why the elements of the solution $\bd\pi(\gamma)$ computed by this approach, generally, do not sum to 1. Similarly, the dual version of (i) proposed above is equivalent to apply a mixed parametrization to a multinomial distribution, because of this, $\bd 1\tr\bd\pi(\gamma)$ = 1 by construction; however, while $\bd\eta$ = $\bd H\log\bd\pi(\gamma)$ is a vector of canonical parameters, the additional log linear constraint $\bd c\tr\log\bd\pi(\gamma)$ = 0 does not correspond to a canonical parameter and leads to a curved exponential family.

Let $\bd F(\gamma)$ = $\bd X\tr[\diag( \bd\pi(\gamma))-\bd \pi(\gamma)\bd \pi(\gamma)\tr]\bd X$; Lemma 2 in \cite{Forcina2012} shows that this matrix is positive definite for all $\bd\pi(\gamma)\in \cg P$. In Appendix 2 of the same paper an alternative to IPF for solving step (i) in either formulations is also outlined. In the present context, it consists in performing the following two steps until convergence:
\begin{enumerate}
\item
given the starting value $\bd\theta(\gamma)^{(s)}$ in the $s$ step, reconstruct the vector of probabilities
$$
\bd\pi(\gamma)^{(s)} = \frac{\exp[\bd X\bd\theta(\gamma)^{(s)}]}{\bd 1\tr\exp[\bd X\bd\theta(\gamma)^{(s)}]};
$$
\item
update parameter estimates
$$
\bd\theta(\gamma)^{(s+1)} = \bd\theta(\gamma)^{(s)} + \left[\bd F(\gamma)^{(s)}\right]^{-1} [\gamma\bd X\tr\bd p-\bd X\tr \bd\pi(\gamma)^{(s)}];
$$
\end{enumerate}
Especially with $r$ large, the above algorithm may be substantially more efficient than IPF which, on the other hand, may be more reliable when elements of $\bd\theta(\gamma)$ get close to the boundary of the parameter space.

To solve step (ii), \cite{KRipf1} propose to start with a pair $\gamma_L^{(0)},\:\gamma_U^{(0)}$ such that $\bd 1\tr\bd\pi(\gamma_L^{(0)})<1$ and $\bd 1\tr \bd\pi (\gamma_U^{(0)}) >1$; at each $\gamma$ adjustment step, compute the midpoint and replace either boundary of the search interval which, in this way, keeps narrowing. Though an analytic expression for the function $f(\gamma)$ = $\bd c\tr\log \bd\pi(\gamma)$ is not available, by exploiting advanced results on the mixed parametrization, Lemma 1 below provides an expression for its derivative and shows that the function is increasing everywhere in its range so that the problem: find $\hat\gamma$ such that $f(\hat\gamma)=0$ may be solved by a Newton algorithm which is considerably faster.

For a given $\hat{\bd\pi} \in \cg M(\bd X)$, let $\bd s$ = $\bd X\tr \hat{\bd\pi}$ and define $\cal G(\bd s)$ to be the set of $\gamma$s for which there exist $\bd p \in \cg P$ such that $\gamma\bd X\tr\bd p$ = $\bd s$. A value of $\gamma\in \cal G(\bd s)$ will be called feasible.
\begin{lemma} For all $\gamma\in \cg G(\bd s)$, $f(\gamma)$ is differentiable with:
$$
\frac{\partial f(\gamma)}{\partial\gamma} = \gamma \bd t\tr \bd F(\gamma)^{-1}\bd t;
$$
because $ \bd F(\gamma)$ is positive definite everywhere, $f(\gamma)$ is strictly increasing.
\end{lemma}
\begin{proof}
See the Appendix.
\end{proof}
Lemma 1 may be used to compute a local inverse of $f(\gamma)$, this, combined with the mixed parametrization, provides the basis of an algorithms for computing the MLE of $\bd\pi\in\cg M(\bd X)$. One can start with $\gamma=1$ which is always feasible and continue with the following steps:
\begin{enumerate}
\item in the $s$th step, given $\gamma^{s}$, use the Newton algorithm to reconstruct the vector of probabilities from the mixed parametrization with mean component $\bd X\tr \bd \pi(\gamma^s)$ =  $\gamma^s\bd X\tr\bd p$ and the canonical parameters $\bd H\log \bd \pi(\gamma^s)$ = $\bd 0$;
\item use Lemma  1 to update $\gamma$:
$\gamma^{s+1} = \gamma^s-f(\gamma^s) / [\gamma^s\bd t\tr \bd F(\gamma^s)^{-1}\bd t]$;
\item iterate until $f(\gamma^s)$ is sufficiently close to 0.
\end{enumerate}
The algorithm usually takes few $\gamma$-updating steps to reach convergence; it may be wise, initially, to shorten the step length when updating $\gamma$  to avoid hitting into a non feasible value. A set of {\sc Matlab} functions to perform the tasks described above are available from the author.
\subsection{The geometry of MLEs}
Recall that a model defined by (\ref{eqEFreg}) alone may be interpreted as a log-linear model for a multinomial distribution and its MLEs are determined only by $\bd t$ = $\bd X\tr\bd p$, the vector of sufficient statistics computed on the sample proportions. For fixed $\bd t$, all vectors of sample proportions $\bd p$ which satisfy $\bd X\tr \bd p$ = $\bd t$ have the same MLEs; this follows from the fact that, within the mixed parametrization, the mapping between $\bd t$ and $\tilde{\bd\theta}$ is one to one

For fixed $\bd s$ = $\bd X\tr\hat{\bd\pi}$, any vector of sufficient statistics $\bd t$ such that $\gamma\bd t$ = $\bd s$, for any$\gamma \in \cg G(\bd s)$, leads to the same MLE under $\cg M(\bd X)$. This, again, may be derived from the mixed parametrization by noting that, for fixed  $\bd\eta$ = $\bd H\bd\lambda$ = $\bd 0$, $\hat{\bd\pi}$ is determined uniquely by $\bd s$. Now consider the space whose coordinates are the sufficient statistics $\bd t$ = $\bd X\tr\bd p$ and define the set $\cg T(\bd s)$
\begin{equation}
\bd t\in\cg T(\bd s)\m{ if } \exists\: \gamma\in \cg G(\bd s) \m { such that } \bd t = \bd s /\gamma.
\label{SufStat}
\end{equation}
In this space $\cg T(\bd s)$ is contained in the straight line joining the origin with the point $\bd s$ and is bounded by the constraint $\gamma\in \cg G(\bd s)$ induced by the fact that $\bd p\in \cg P$. It may be of interest to clarify the relation between (\ref{SufStat}) and equation (3.3) in Efron which applies to any curved exponential family defined by a smooth mapping $\bd\theta(\bd\phi)$ where the parameter vector $\bd\phi\in \cg F\subseteq\Re^q$ with $q<k$. Consider the likelihood equation for the multinomial distribution as a (general) curved exponential family, together with the equation obtained by differentiating the non homogeneous constraint $F(\bd\theta)$ with respect to $\bd\phi$, by the chain rule
$$
\left\{
\begin{array}{l}
\frac{\partial\bd\theta\tr}{\partial\bd\phi}\bd X\tr (\bd p-\bd \pi )=  \bd A\tr (\bd t - \bd s) = \bd 0\\
\frac{\partial\bd\theta\tr}{\partial\bd\phi} \frac{\partial}{\partial\bd\theta}
\log[\bd 1\tr\exp(\bd X\hat{\bd\theta})] = \bd A\tr \bd s = \bd 0,
\end{array}
\right.
$$
like (3.3) in \cite{Efron78}, the first equation indicates that the vector $\bd t - \bd s$ must be orthogonal to the columns of the jacobian matrix $\bd A$ while the second says that $\bd s$ itself must satisfy the same orthogonality property. The two equations together imply (\ref{SufStat}) because, if both $\bd s$ and $\bd t-\bd s$ must be orthogonal to the columns of $\bd A$, $\bd t$ must be proportional to $\bd s$.
\begin{example}
Consider again the same design matrix as in Example 1.
Because $\bd \pi\in \cg P$, the fourth dimension is redundant and sample proportions can be plotted in three dimensions. For each $\bd p$ on a grid within $\cg P$, the resulting $\hat{\bd\pi}$ is plotted in Figure 1, left panel, together with the corresponding value of $\bd s$ = $\bd X\tr\hat{\bd\pi}$, right panel. Each set of points outline a corresponding two dimensional surfaces.
\begin{figure}
\centering
\makebox{\includegraphics[width=14cm,height=6cm]{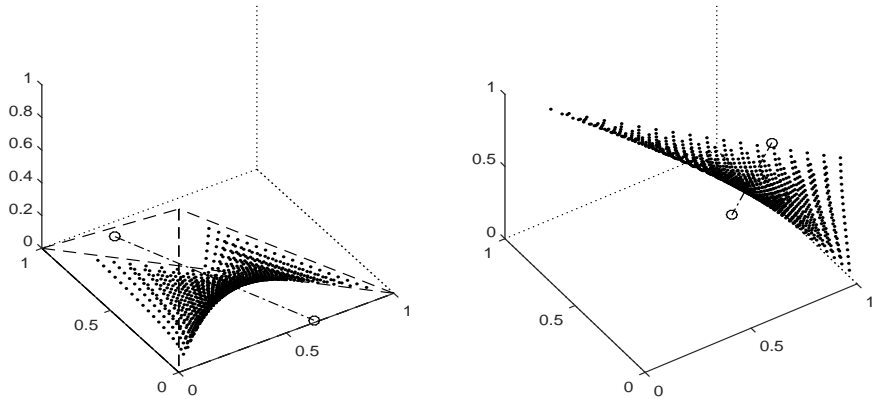}}
\caption{\label{Fig2}For each vector of sample proportions belonging to a grid with step size 0.05 in each direction within $\cg P$, the first three coordinates of $\hat{\bd\pi}$ are plotted on the left panel; the corresponding values of $\bd s$ are plotted on the right hand side. For a specific value of $\hat{\bd\pi}$, the linear set $\gamma\bd X\tr \hat{\bd\pi}$, $\gamma\in\cg G(\bd s)$ is plotted on the right hand side by a dash-dot line with circles marking end points. The corresponding set $\cg T(\bd s)$ is also plotted on the left panel.}
\end{figure}
\end{example}
The pair $(\bd s,\:\gamma)$ determines a single element of $\cg T(\bd s)$; among these, there are two special points worth mentioning: (i) the point with $\bd t=\bd s$ $\Leftrightarrow$ $\gamma=1$ which arises in the very unlikely event that $\tilde{\bd\pi}$, the MLE under the log linear model, belongs to $\cg M(\bd X)$;
(ii) the point corresponding to the largest feasible adjustment factor, say, $\gamma_U$, which lies on the boundary of the sample space closest to the origin. The results on the asymptotic distribution of $\hat\gamma$ in the next section may be used to determine an acceptance region for $\cg M(\bd X)$ along the line $\bd t$ = $\bd s/\gamma$.

The following Proposition shows that, for fixed $n$ and $\bd s$, the likelihood is an inverse function of $\gamma$, thus it is maximum when $\bd t$ = $\bd s/\gamma_U$.
\begin{proposition}
\label{P3}
For $n,\:\bd s$ fixed, the log-likelihood $\bd y\tr\log \hat{\bd\pi}$ achieves its maximum when $\hat\gamma$ = $\gamma_U$.
\end{proposition}
\begin{proof}
Replace $\log\hat{\bd \pi}$ with $\bd X\hat{\bd\theta}$, and use (\ref{SufStat})
$$
L = \bd y\tr \log \hat{\bd\pi} = n\bd p\tr \bd X \hat{\bd\theta} = n\bd t\tr\hat{\bd\theta} = \bd s\tr\hat{\bd\theta}/\hat\gamma.
$$
Because the numerator is negative, the maximum is achieved when $\hat\gamma$ is largest.
\end{proof}
\section{Hypothesi testing}
When assessing the performance of multiplicative models, it seems reasonable to test first whether the log-linear model (\ref{eqEFreg}) fits well and, only when this test is passed, consider fitting $\cg M(\bd X)$.  Let
$$
D_L = 2\bd y'[\log(\bd y/n) - \log\tilde{\bd\pi}]; \quad D_M = 2\bd y'[\log(\tilde{\bd\pi}) - \log\hat{\bd\pi}].
$$
It follows from basic results on log-linear models that the asymptotic distribution of $D_L$ is $\chi^2_{r-k-1}$. Considering that $\cg M(\bd X)$ may be obtained from (\ref{eqEFreg}) by imposing the smooth non-linear constrain $\log[\bd 1\tr \exp(\bd X\bd\theta)]=0$, it follows from general results on the asymptotic distribution of the likelihood ratio that, when the sample size goes to infinity, $D_M$ $\sim$ $\chi^2_{1}$. It may be of interest to examine the behaviour of $D_M$ as a function of $\gamma$ for $t\in \cg T(\bd s)$; suppose that $\bd t$ = $\bd s/\gamma$ is the observed vector of sufficient statistics and $\bd\theta(\gamma)$ is the MLE under the log-linear model (\ref{eqEFreg}). Recall that, in Lemma 1, the scaling factor for probabilities estimated under the log-linear model was called $f(\gamma)$ and shown to be increasing everywhere for fixed $\bd t$. Let $g(\gamma)$ denote the scaling factor as a function of $\gamma$ for fixed $\bd s$; in the proof of Lemma 1 it is shown that $g(\gamma)$ is strictly decreasing everywhere in $\cg G(\bd s)$. By simple calculations,
$$
D_M(\gamma)= 2 n\bd p\tr\left[\bd X\bd\theta(\gamma)-g(\gamma) -\bd X\hat{\bd\theta}) \right]\\
= 2n\left[\bd s\tr(\bd\theta(\gamma) -\hat{\bd\theta})/\gamma - g(\gamma)\right].
$$
A plot of $D_M(\gamma)$ and $g(\gamma)$ as functions of $\gamma$ are given in Example \ref{ASind} below where it can be seen that $D_M(\gamma)$ is approximately quadratic in an interval around $\gamma=1$ while $g(\gamma)$ is 0 at 1 and its slope get very steep near the boundaries of $\cg G(\bd s)$.

Within the more general context of Relational models, \cite{KRarxiv16} study the asymptotic distribution of the Bergman statistic and show that it is asymptotically equivalent to the distribution of $D_L + D_M$; they also show that the distribution of the parameter estimates converge to normality. Similar conclusions concerning the degrees of freedom for the LRs statistics above can also be derived from the results of \cite{KRD12}, section 2.

The problem considered here may also be seen as a special instance of the one studied by \cite{AitchSilvey58} who consider testing a set of smooth non linear equality constraints. In the context of multiplicative models  there is just one constraint of the form
$$
F(\bd\theta) = \log\left[\bd 1\tr\exp(\bd X\bd\theta)\right] = 0;
$$
because the elements of the jacobian $\bd H$ = $\partial F / \partial\bd\theta $ = $\bd X\tr\bd\pi$  are strictly positive and differentiable, it follows \citep[][Therem 2, p. 824]{AitchSilvey58} that, if model $\cg M(\bd X)$ is true, the estimate of the Lagrange multiplier $\alpha$ and that of the adjustment factor $\gamma$ have both an asymptotic normal distribution
\begin{equation}
\hat\alpha  \sim N(0,\sigma^2_{\hat\alpha}),\quad
\hat\gamma  \sim N(1,\sigma^2_{\hat\gamma}) \m{ where }\sigma^2_{\hat\alpha} = [n \hat{\bd\pi}\tr\bd X \bd F(\hat\gamma)^{-1} \bd X\tr \hat{\bd\pi}]^{-1}
\label{ASint}
\end{equation}
and $\sigma^2_{\hat\gamma}$ = $\sigma^2_{\hat\alpha}/\gamma^4$,
where both variances have been computed by the Delta method. These results lead to two additional statistics which are asymptotically equivalent to $D_M$, $L$ = $\hat\alpha^2/\sigma^2_{\hat\alpha}$ and $G$ = $(\hat\gamma-1)^2/\sigma^2_{\hat\gamma}$.

The asymptotic distribution of $\hat\gamma$ may be used to determine a kind of acceptance region around $\gamma=1$ which corresponds to an interval along the line $\bd t$ = $\bd s/\gamma$ of sample statistics leading to the same MLE: only points $\bd t$ sufficiently close to $\bd s$ are compatible with $\cg M(\bd X)$.

\begin{example} \label{ASind}
Table \ref{tab:2} contains the frequency distribution of baskets containing only one of the following items: biscuits, milk, tomato sauce. The original data-set, described in \cite{Giudici02}, is about the presence/absence of 20 items in each of the over 46,000 baskets passing at the counters of a food store in a given period. The data used here is obtained by conditioning to the remaining 17 items being absent, a restriction satisfied by about 13\% of the total.
\begin{table}[ht]
\caption{\label{tab:2} Frequency distribution of baskets with at most three items, including: b = biscuits, m = milk, t = tomato souce}
\centering
\fbox{
\begin{tabular}{lrrrrrrr} \hline
Composition & t & m & mt & b  & bt & bm & bmt \\
 \hline
number & 374 &  3684 &  233  & 991 & 41 &  607 & 46 \\
\hline
\end{tabular}}
\end{table}
The $\bd X$ matrix for the QI and AS independence model described, for instance, by \cite{KRJMA16}, is given below
$$
\bd X\tr = \begin{pmatrix}
     0  &   0 &  0 &  1 & 1 & 1 & 1 \\
     0  &   1 &  1 &  0 & 0 & 1 & 1 \\
     1  &   0 &  1 &  0 & 1 & 0 & 1
\end{pmatrix}.
$$
The QI model cannot be rejected: deviance of 9.14 on 4 degrees of freedom with a $p$-value of 0.055. Given QI, the model of AS independence fits almost perfectly with $D_M$ = 0.02 and $\hat\gamma$ = 0.9994.  In this case $\hat\gamma$ is well inside the acceptance region based on the asymptotic normal distribution (0.9923, 1.0077).  Clearly these data concern a very special subset of customers. For $\bd s$ fixed to the value determined by these data, the values of $D_M(\gamma)$ on an interval around $\bd s$ is plotted in Figure \ref{Fig3} together with the 95\% likelihood ratio based confidence interval for $\gamma$. The plot of $g(\gamma)$ is approximately linear around $\gamma=1$ but curvature increases when $\gamma$ approaches the boundaries of $\cg G(\bd s)$ which, in this instance, are 0.765 and  1.162).
\begin{figure}
\centering
\makebox{\includegraphics[width=16cm,height=6cm]{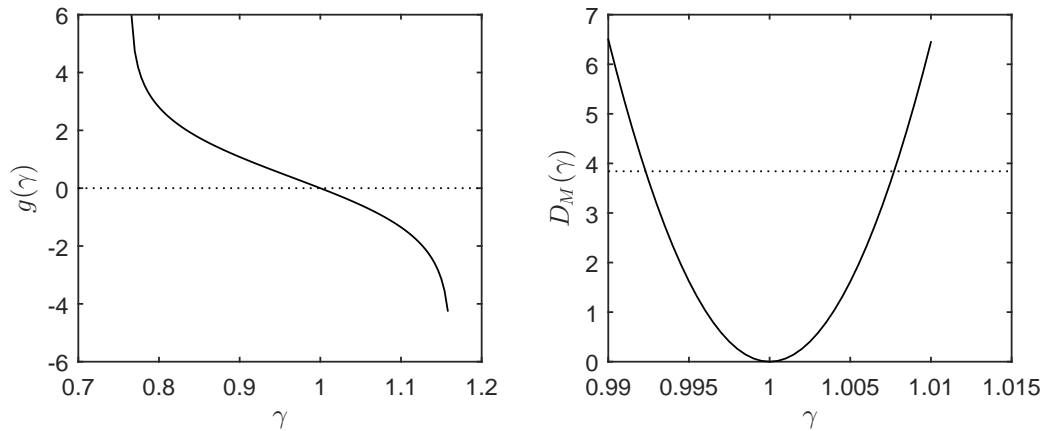}}
\caption{\label{Fig3} For $n$ and $\bd s$ fixed to the value in the basked data, the scaling factor $g(\gamma)$ is plotted on the left panel and that of  $D_M(\gamma)$ is plotted on the right panel for a small range of $\gamma$ around 1.}
\end{figure}
\end{example}
\subsection{A simulation study}
A small simulation study was performed to compare the accuracy of the asymptotic approximations of the LR, the Lagrange multiplier test and that of the adjustment factor. The AS independence model fitted in example \ref{ASind} above was taken as the null distribution. Sample sizes of 200, 1,000 and 5,000 were used with 40,000 replications for each sample size. The estimated rejection rates for the three statistics are displayed in Table \ref{tab:3}.
\begin{table}[htb]
\caption{\label{tab:3} Estimated rejection rates under $H_0$ $\bd\pi \in \cg M(\bd X)$}
\centering
\fbox{
\begin{tabular}{lrrrrrrrrr} \hline
sample size & \multicolumn{3}{c}{10\%} &  \multicolumn{3}{c}{5\%} & \multicolumn{3}{c}{1\%} \\
$n$ & $D_M$ & $L$ & $G$ & $D_M$ & $L$ & $G$ & $D_M$ & $L$ & $G$ \\
 \hline
  200&   10.05 & 10.05 & 9.78 & 5.09 & 5.00 & 4.97 & 0.95 & 0.96 & 1.14 \\
1,000&    9.99 &  9.95 & 9.93 & 4.97 & 4.98 & 4.92 & 1.05 & 1.05 & 1.04 \\
5,000&   10.01 &  9.99 & 9.98 & 4.94 & 4.95 & 4.96 & 0.99 & 0.97 & 1.03 \\
\hline
\end{tabular}}
\end{table}
On the whole, by looking at averages of absolute errors, the rejection rates of $D_M$ are the closest to their nominal value while those of $G$ are the worst. By the same criterion, the distributions of the three statistics get closer to $\chi^2_1$ as the sample size increases.
\section*{Acknowledgments}
The author would like to thank A. Klimova and T. Rudas for sharing ideas concerning Relational models and for several very enlightening discussions and A. Salvan for comment on the nature of the curved exponential family.
\appendix
\section*{Appendix}
\subsection*{Multinomial and Poisson as exponential families}
Let $\bd v$ $\sim$ Mn$(n,\bd\pi)$ where $\bd\pi$ has dimension $q$; a multivariate logistic transform of $\bd\pi$ may be defined as $\log\bd\pi$ = $\bd G\bd\lambda-\bd 1_q\log[\bd 1_q\tr\exp(\bd G\bd\lambda)]$, where $\bd\lambda$ is a vector of canonical parameters determined by $\bd G$, an arbitrary $q \times (q-1)$ matrix of full rank whose columns do not span the unitary vector.
The kernel of the log of the probability distribution may be written as
$$
\bd v\tr\bd\log\bd\pi = \bd v\tr\bd G\bd \lambda - n\log[\bd 1_q\tr\exp(\bd G\bd\lambda)],
$$
then $\bd t$ = $\bd G\tr\bd v$ is the vector of sufficient statistics and $K(\bd\lambda)$ = $n\log[\bd 1\tr\exp(\bd G\bd\lambda)]$.

To derive an explicit  expression for $\bd\lambda$, let $\bd R$ = $\bd I_q-\bd 1_q\bd 1_q\tr/q$ and
$$
\bd D = (\bd G\tr \bd R \bd G)^{-1} \bd G\tr \bd R; \quad \Rightarrow \bd D\bd G =\bd I_q,\quad
\bd D\bd 1_q=\bd 0_q,
$$
then $\bd\lambda$ = $\bd D\log\bd\pi$ ia s vector of $q-1$ canonical parameters. To see why the coefficient of any linear constraint on canonical parameters must sum to 0, note that $\bd D\bd 1_q$ = $\bd 0_{q-1}$.
To introduce linear restrictions on $\bd\lambda$, assume that $\bd G$ is partitioned as $(\bd X\:\: \bd Z)$, where $\bd Z$ is such that  $\bd Z\tr \bd R\bd X$ = $\bd 0$, let also $\bd H$ = $(\bd Z\tr \bd R\bd Z)^{-1}\bd Z\tr\bd R$; now define $\bd\eta$ = $\bd H\bd\lambda$. Then the model $\bd\lambda$ = $\bd X\bd\theta$ is equivalent to assume that $\bd\eta$ = $\bd 0$.

If, instead, the elements of $\bd v$ were distributed as $q$ independent Poisson variables, the kernel of the log of the probability distribution would be
$$
\bd v\tr\log\bd \mu-\bd 1\tr\bd\mu = \bd y\tr\bd \lambda -K(\bd\lambda),
$$
where $\bd\lambda$ = $\log\bd\mu$ and $K(\bd\lambda)$ = $\bd 1\tr \exp(\bd\lambda)$
\subsection*{Proof of Lemma 1}
To differentiate $f(\gamma)$ = $\log[\bd 1\tr\exp(\bd X\bd\theta(\bd\gamma))]$ note that (\ref{eqajfac}) implies $\bd\tau(\gamma)$ = $\bd X\tr\bd \pi(\gamma)$ = $\gamma\bd X\tr\bd p$. By the chain rule
$$
\frac{\partial f(\gamma)}{\partial\gamma} =
\frac{\partial f/\gamma)}{\partial\bd\theta(\gamma)\tr}
\frac{\partial \bd\theta(\gamma)}{\partial \bd\tau(\gamma)\tr}
\frac{\partial \bd\tau(\gamma)}{\partial\gamma} =
\frac{\exp(\bd X\bd\theta(\gamma))\tr}{\bd 1\tr\exp(\bd X \bd\theta(\gamma))} \bd X \frac{\partial\bd\theta(\gamma)}{\partial\bd\tau(\gamma)\tr}\bd X\tr\bd p.
$$
The result follows because, by construction, $\bd X\tr \exp(\bd X\bd\theta(\gamma)) /[\bd 1\tr\exp(\bd X\bd\theta(\gamma))]$ = $\bd \tau(\gamma)$ = $\gamma \bd X\tr \bd p$ and
$$
\frac{\partial\bd\theta(\gamma)}{\partial \bd\tau(\gamma)\tr} =
\left(\frac{\partial\bd\tau(\gamma)}{\partial \bd\theta(\gamma)\tr}\right)^{-1} =
\bd X\tr\frac{\partial\bd \pi(\gamma)}{\partial (\bd X\bd\theta(\gamma))\tr}\bd X = \bd F(\gamma).
$$
Differentiation of the function $g(\gamma)$ is similar, except that, because $\bd\tau(\gamma)$ = $\bd s/\gamma$, the last component in the derivative is $-\bd s/\gamma^2$.

\bibliographystyle{apalike}
\bibliography{Oviedo}
\end{document}